\documentclass[11pt]{article}
\usepackage{amscd}
\usepackage{amsfonts}
\usepackage{amsmath,amsthm,hyperref}
\usepackage[all]{xy}
\usepackage{amsmath,amsthm,hyperref}
\usepackage{amsmath,amssymb,amsthm,latexsym}
\usepackage{amscd}
\newtheorem{theorem}{Theorem}[section]
\newtheorem{proposition}[theorem]{Proposition}
\newtheorem{definition}[theorem]{Definition}
\newtheorem{corollary}[theorem]{Corollary}

\theoremstyle{remark}
\newtheorem{remark}[theorem]{Remark}
\newtheorem{example}[theorem]{\bf Example}

\begin{document}
\title{\bf{Generalized polar transforms of spacelike isothermic surfaces}}
\author{Peng Wang\footnote{ Supported by Program for Young Excellent Talents in Tongji University and the Tianyuan Foundation of China, grant 10926112.} }
\date{}
\maketitle

\begin{center}
{\bf Abstract}
\end{center}
In this paper, we generalize the polar transforms of spacelike isothermic surfaces in $Q^4_1$ to n-dimensional pseudo-Riemannian space forms $Q^n_r$. We show that there exist $c-$polar spacelike isothermic surfaces derived from a spacelike isothermic surface in $Q^n_r$, which are into $S^{n+1}_r(c)$, $H^{n+1}_{r-1}(c)$ or $Q^n_r$ depending on $c>0,<0,$ or $=0$. The $c-$polar isothermic surfaces can be characterized as generalized $H-$surfaces with null minimal sections. We also prove that if both the original surface and its $c-$polar surface are closed immersion, then they have the same Willmore functional. As examples, we discuss some product surfaces and compute the $c-$polar transforms of them. In the end, we derive the permutability theorems for $c-$polar transforms and Darboux transform and spectral
transform of
isothermic surfaces. \\




\section{Introduction}

As a classical topic in differential
geometry, isothermic surfaces are well known as that they admit various kinds of transforms to derive new isothermic surfaces, such as the Christoffel transform
(the dual
isothermic surface), the
spectral transforms (also known as T-transforms, Bianchi
transforms or Calapso transforms), and the Darboux transforms. And there are also permutability theorems relating
them (\cite{Bur},\cite{Jer}). However, it was revealed only in the past 20 years \cite{CGS, BHPP,FP, Bur, BDPT} that,
there is a structure of integrable system underlying the theory
about isothermic surfaces, which has its root in the transforms of isothermic surfaces.

In recent years, the integrable system theory of isthermic surfaces is used to other cases to study transforms of isothermic surfaces. Dussan \cite{Du}, Fujioka and Inoguchi \cite{FI}, began to study
isothermic surfaces in pseudo-Riemannian space forms. And Zuo et al
\cite{ZCC} generalized the Darboux transform of isothermic
surfaces to the pseudo-Riemannian space forms using the methods
developed by Burstall in \cite{Bur} and Bruck et al in
\cite{BDPT}.

On the other hand, in \cite{Ma-W1,Ma-W2}, Ma and Wang found that there exists a natural transform of spacelike surfaces in $Q^4_1$, the
conformal compactification of the 4-dimensional Lorentzian space
forms. The key observation is
that in this codim-2 case, the normal plane at any point is
Lorentzian. The two null lines $[L],[R]$ in this plane define two
conformal maps into $Q^4_1$, called the \emph{ left} and the
\emph{right polar surface}, respectively. Conversely, $Y$ is also
the right polar surface of $[L]$, and the left polar surface of
$[R]$ (when $[L]$ and $[R]$ are immersions). It is proved in \cite{Ma-W1,Ma-W2} that these transforms preserve
the Willmore property and isothermic property.

In this paper we generalize this produce to the higher codimension
by a key observation that the conformal normal bundle of a spacelike isothermic
surface is flat, which allows us to derive a parallel section of the normal bundle conformal to the isothermic surface. So we can define the polar transforms
locally. A new isothermic surface produced in this way is neither
the spectral transform nor the Darboux transform of $[Y]$. It can be looked as  the generalization of the Christoffel transform for isothermic surface, see Section 3 for details. It is also interesting to see that all $c-$polar surfaces are generalized $H-$surfaces with null minimal sections (when c=0, it should be looked as surfaces in $\mathbb{R}^{n+2}_{r+1}$, see Section 3 for details).
We also
discuss the moduli space of $c-$polar transforms of a spacelike
isothermic surface in $Q^n_r$ in the meaning of isometric equivalence. See Theorem~\ref{thm-polar}.

Since $c-$polar transforms can be viewed as generalized Christoffel transforms, it is natural to expect that $c-$polar
transforms commute with the spectral transform and the Darboux
transform. Similar to \cite{Ma-W2}, we derive two of such
\emph{permutability theorems}, see
Theorem~\ref{thm-commu1} and Theorem~\ref{thm-commu2}.

This paper is organized as follows. In Section~2  we
review the main theory of spacelike surfaces both in the isometric case and in the conformal case, together with the relations between them. The definition and basic properties of
isothermic surfaces are also discussed here. Then we introduce
the $c-$polar transforms of spacelike isothermic surfaces and the description of them as the main theorem in Section~3.
Finally, we provide a brief proof of the permutability theorems between $c-$polar transforms and the spectral transforms, the Darboux
transforms in Section~4.

\section{Surface theory for spacelike isothermic surfaces}

In this section we first give the surface theory in space forms. Then we derive the surface theory in the conformal geometry of surface theory. In the end the relations between the isometric variants of surfaces and the conformal variants of a surface is obtained.

\subsection{Spacelike isothermic surfaces in $N^n_r(c)$}

Let $\mathbb{R}^m_s$ be the space $\mathbb{R}^m$ equipped with the
quadric form
\[
\langle x,x\rangle=\sum^{m-s}_{1}x^2_i-\sum^m_{m-s+1}x^{2}_i.
\]
Let $N^n_r(c)$ denote the n-dimensional pseudo-Riemannian space form with constant curvature $c$, $c=0,1,-1$,  which has a pseudo-Riemannian metric of signature $(n-r,r)$, $n-r\geq2,r\geq0$. Then for a point $
x\in N^n_r(c),$

$$\left\{\begin{split}
&x\in  N^n_r(c)=R^n_r\subset \mathbb{R}^n_r,\ c=0,\\
&x\in  N^n_r(c)=S^n_r\subset\mathbb{R}^{n+1}_r,\ c=1,\\
&x\in  N^n_r(c)=H^n_r\subset\mathbb{R}^{n+1}_{r+1},\ c=-1.\\
\end{split}
\right.
$$
Let  $x:M\rightarrow N^n_r(c)$ be a spacelike surface with complex coordinate $z$ and metric $|dx|^2=e^{2\omega}|dz|^2$. The structure equation can be given as
\begin{equation}\label{eq-moving-1}
\left\{\begin {array}{lllll}
x_{zz}=2\omega_{z}x_z+\Omega,\\
x_{z\bar{z}}=\frac{1}{2}e^{2\omega}H-\frac{c}{2}e^{2\omega}x,\\
n_{\alpha z}=D_{z}n_\alpha-\langle n_{\alpha}, H\rangle x_{z}-2e^{-2\omega}\langle n_{\alpha}, \Omega\rangle  x_{\bar{z}},\ \alpha=3,\cdots,n.\\
\end {array}\right.
\end{equation}
Here  $\{n_\alpha\}$ is a orthonormal frame of the normal bundle with $\langle n_\alpha,n_\alpha\rangle=\varepsilon_{\alpha}=\pm1$ and $D_z$ denotes the normal connection. The $\Gamma(T^{\perp}M\otimes\mathbb{C})-$valued 2-form $\Omega dz^2=\sum_{\alpha=3}^{n}\Omega^{\alpha}n_{\alpha}dz^2$ is  the vector-valued {\em Hopf differential}.
 We also assume that the mean curvature vector $H=\sum_\alpha h^\alpha n_\alpha$.
Then we have the integrable equations:
\begin{equation}\label{eq-G-C-R}
\left\{\begin {array}{lllll}
\langle H,H\rangle-K+c=4e^{-4\omega}\langle \Omega,\bar\Omega\rangle,\\
D_zH=2e^{-2\omega}D_{\bar{z}}\Omega,\\
R^{D}_{\bar{z}z}n_{\alpha}:=D_{\bar{z}}D_{z}n_\alpha-D_{z}D_{\bar{z}}n_\alpha
=2e^{-2\omega}(\langle n_{\alpha},\Omega\rangle\bar{\Omega}-\langle
n_{\alpha},\bar{\Omega}\rangle\Omega).\\
\end{array}\right.
\end{equation}

To see the fundamental forms of $x$, we assume
that the two fundamental forms of $x$ are of the form
\begin{equation*}
I=e^{2\omega}(du^2+dv^2),\ II=\sum_{\alpha}(b^{\alpha}_{11} du^2+2b^{\alpha}_{12}dudv+b^{\alpha}_{22} dv^2)n_{\alpha}, z=u+iv.
\end{equation*}
Comparing with the coefficients in \eqref{eq-moving-1}, we have that
$$b^{\alpha}_{11}=(e^{2\omega}h^{\alpha}+\Omega^{\alpha}+\overline{\Omega^{\alpha}}),\ b^{\alpha}_{22}=(-e^{2\omega}h^{\alpha}+\Omega^{\alpha}+\overline{\Omega^{\alpha}}),\ b^{\alpha}_{12}=i(\Omega^{\alpha}-\overline{\Omega^{\alpha}}).$$

Now we define the isothermic surfaces as follow.
\begin{definition}\cite{BPP}
Let $x:M\rightarrow N^{n}_{r}(c)$ be a spacelike surface. It is called isothermic if around each
point of $M$ there exists a complex coordinate $z$ such that the vector-valued Hopf differential $\Omega$ is real-valued. Such a coordinate $z$ is called an adapted coordinate.
\end{definition}

\begin{remark}  (Comparing with the classical definition, \cite{BDPT}, \cite{Bur}, etc.) Here $\Omega$ is real-valued if and only if all of $\Omega^{\alpha}$ are real-valued functions.
The real-valued property of $\Omega$, together with the Ricci equations in
\eqref{eq-G-C-R}, shows that the normal bundle of $x$ is flat. This is
an important property of isothermic surfaces, which guarantees
that all shape operators commute and the curvature lines could
still be defined. Let $\{n_{\alpha}\}$ be a parallel orthonormal frame of the flat normal bundle.
Since $\Omega^{\alpha}=\overline{\Omega^{\alpha}}$, the two fundamental
forms of an isothermic surface are of the form
\begin{equation}\label{iso}
I=e^{2\omega}(du^2+dv^2),\  II=\sum_{\alpha}(b^{\alpha}_{11} du^2+b^{\alpha}_{22} dv^2)n_{\alpha},
\end{equation}
with respect to the parallel normal frame $\{n_{\alpha}\}$. Then
$(u,v)$ are the conformal curvature line parameters. So our definition of spacelike isothermic surfaces coincides with the classical definition:
 
{\em $x$ is isothermic if and only if there exists a pair of coordinate $(u,v)$ and  parallel normal frame $\{n_{\alpha}\}$ such that the two fundamental forms are of the form in \eqref{iso}.}
\end{remark}

\begin{remark} Another significant property of isothermic surfaces is that they are conformally invariant,
which suggests that it may be better to treat them using a conformally invariant frame.
This is the main topic in the next subsection.
\end{remark}

\subsection{Spacelike isothermic surfaces in $Q^{n}_{r}$}

Let $\mathbb{R}^m_s$  be the pseudo Euclidean space $(\mathbb{R}^m,\langle,\rangle)$ as above.
We denote by $C^{m-1}_s$ the light cone of $\mathbb{R}^m_s$, and by
\[
Q^n_r=\{\ [x]\in\mathbb{R}P^{n+1}\ |\ x\in C^{n+1}_r\setminus \{0\}
\}
\]
the projectived light cone, with the standard projection
$\pi:C^{n+1}_{r}\setminus\{0\}\rightarrow Q^n_r$. $Q^n_r$ can be equipped with a $(n-r,r)-$type pseudo-Riemannian metric
induced from the projection $\pi:S^{n-r}\times S^r\rightarrow Q^n_r$ . Here
$$S^{n-r}\times
S^r=\{x\in\mathbb{R}^{n+2}_{r+1}\ |\
\sum^{n-r+1}_{i=1}x^{2}_{i}=\sum^{n+2}_{i=n-r+2}x^2_{i}=1\}\subset
C^{n+1}_r\setminus\{0\}
$$
is endowed with a $(n-r,r)-$type pseudo-Riemannian metric
$g(S^{n-r})\oplus (-g(S^r))$, where $g(S^{n-r})$ and $g(S^r)$ are
standard metrics on $S^{n-r}$ and $S^r$. So there is a conformal
structure of $(n-r,r)-$type pseudo-Riemannian metric $[h]$ on
$Q^n_r$. It is well known that
the conformal group of $(Q^n_r,[h])$ is exactly the orthogonal group
$O(n-r+1,r+1)/\{\pm1\}$, which keeps the inner product of
$\mathbb{R}^{n+2}_{r+1}$ invariant and acts on $Q^n_r$ by
$
T([x])=[xT],\ T\in O(n-r+1,r+1).$
Then $N^n_r,\ c=0,1,-1$, can be conformally embedded as a proper subset of $Q^n_r$ via
\begin{equation}x \mapsto\left(\frac{1-\langle x,x\rangle}{2},x,\frac{1+\langle x,x\rangle}{2}\right).
\end{equation}
Note that when $c=1$, the first coordinate above equals to zero, and when $c=-1$ the last coordinate equals to zero. So in both cases we may look it as being a map into $Q^n_r$.

The basic conformal surfaces theory shows that for a spacelike surface $y:M\rightarrow Q^{n}_{r}$ with local coordinate $z$, there exists a local canonical lift $Y:U\rightarrow
C^{n+1}_{r}\setminus\{0\}$ of $y$ such that $|{\rm d}Y|^2=|{\rm d}z|^2$ in an open subset $U$ of $M$. We denote
\begin{equation}
V={\rm Span}\{Y,{\rm Re}(Y_z),{\rm Im}(Y_z),Y_{z\bar{z}}\}
\end{equation}
the conformal tangent bundle, and $V^{\perp}$ its orthogonal complement or the conformal normal bundle.
Let $D$ denote the normal connection, i.e. the induced connection
on the bundle $V^{\perp}$. Let $\psi\in \Gamma (V^{\perp})$ denote an arbitrary section.
Then the structure equation of $y$ can be derived as follows
(\cite{BPP},\cite{Ma1}).
\begin{equation}\label{eq-moving}
\left\{\begin {array}{lllll}
Y_{zz}=-\frac{s}{2}Y+\kappa,\\
Y_{z\bar{z}}=-\langle\kappa, \kappa\rangle Y+\frac{1}{2}N,\\
N_{z}=-2\langle\kappa, \kappa\rangle  Y_{z}-sY_{\bar{z}}+2D_{\bar{z}}\kappa,\\
\psi_{z}=D_{z}\psi+2\langle\psi,D_{\bar{z}}\kappa\rangle Y-2\langle\psi, \kappa\rangle Y_{\bar{z}}.
\end {array}\right.
\end{equation}
The first equation above defines two basic invariants $\kappa$ and $s$ dependent on
$z$, called the \emph{conformal Hopf differential} and the
\emph{Schwarzian derivative} of $y$, respectively (see
\cite{BPP},\cite{Ma1}).

The conformal Hopf differential defines the conformal invariant  metric
\begin{equation} g:=\langle\kappa,\bar\kappa\rangle|dz|^2,
\end{equation}
and the well-known {\em Willmore functional}
\begin{equation} W(y):=2i\int_M\langle\kappa,\bar\kappa\rangle dz\wedge d\bar{z}.
\end{equation}
For more details on this subject, we refer to \cite{BPP,Ma1,Ma2006,Ma-W1}, etc.

The conformal Gauss, Codazzi and Ricci equations as integrable
conditions are:
\begin{equation}\label{eq-G}
\frac{1}{2}s_{\bar{z}}=3\langle\kappa,D_{z}\bar{\kappa}\rangle+\langle D_{z}\kappa,\bar{\kappa}\rangle,\\
\end{equation}
\begin{equation}\label{eq-C}
{\rm
Im}\big(D_{\bar{z}}D_{\bar{z}}\kappa+\frac{1}{2}\bar{s}\kappa\big)=0,
\end{equation}\begin{equation}\label{eq-R}
R^{D}_{\bar{z}z}\psi:=D_{\bar{z}}D_{z}\psi-D_{z}D_{\bar{z}}\psi
=2(\langle \psi,\kappa\rangle\bar{\kappa}-\langle
\psi,\bar{\kappa}\rangle\kappa).\end{equation}

Spacelike {\em isothermic} surfaces are also defined as of real-valued conformal Hopf differential with respect to some complex coordinate (also called the {\em adapted coordinate}).

To see this, we give the relations between the isometric invariants and conformal invariants of a spacelike surface
 $x:M\rightarrow N^n_r(c)$ as above.  We may assume that $c=0$. The other cases are similar. We denote
\begin{equation}X=\left(\frac{1-\langle x,x\rangle}{2},x,\frac{1+\langle x,x\rangle}{2}\right).
\end{equation}
Then $y=[X]$ is the corresponding surface into $Q^n_r$. It is direct to obtain a canonical lift $Y=e^{-\omega}X$ w.r.t. $z$. So
\begin{equation}\left\{\begin{split}
&Y=\frac{1}{2}e^{-\omega}(1-\langle x,x\rangle,2x,1+\langle x,x\rangle),\\
&Y_z=-\omega_{z}Y+e^{-\omega}(-\langle x,x_z\rangle,x_z,\langle x,x_z\rangle),\\
&N=e^{\omega}(-1-\langle H,x\rangle,H,1+\langle H,x\rangle)-2\omega_{\bar{z}}Y_z-2\omega_zY_{\bar{z}}-2|\omega_z|^2Y,\\
&\psi_\alpha=(-\langle n_\alpha,x\rangle,n_\alpha,\langle n_\alpha,x\rangle)+\langle n_\alpha,H\rangle Y.\\
\end{split}\right.\end{equation}
Then we see that the normal connection of $\{\psi_\alpha\}$ is the same as  the normal connection of $\{n_\alpha\}$.
we also calculate the Schwarzian and conformal Hopf differential as
\begin{equation}\label{eq-k-tr}
s=2\omega_{zz}-2\omega_z^2+2\langle\Omega,H\rangle,\ \kappa=e^{-\omega}(-\langle \Omega,x\rangle,\Omega,\langle \Omega,x\rangle)+\langle\Omega, H\rangle Y.
\end{equation}
Now we derive that $\kappa$ is real-valued if and only if $\Omega$ is real-valued.

\section{Generalized $c-$polar transforms of spacelike isothermic surfaces}

In \cite{Ma-W2}, they gave the definition of polar transforms of spacelike isothermic surfaces in $Q^4_1$. Here we will show that such transforms can be defined in a more generalized case for isothermic surfaces.

\begin{definition}Let $y:M\rightarrow Q^{n}_{r}$  be a spacelike isothermic surface. Suppose that $\psi$ is a parallel section of the normal bundle of $y$ with length $c$, i.e. $\psi\in \Gamma(V^{\perp})$, $D_z\psi=0$, and $\langle\psi,\psi\rangle=c$.
 Then we call $\psi:M\rightarrow N^{n+1}_r(\frac{1}{c}) $ a $c-$polar transform of $y$ when $c\neq0$, and  $[\psi]:M\rightarrow Q^{n}_{r}$ a $0-$polar transform of $y$ when $c=0$.
\end{definition}

Then, the polar transforms defined in \cite{Ma-W2}
 are exactly  special cases  of the $0-$polar transforms. We also note that the polar transforms can also be characterized as below:
\begin{proposition}$\psi:M\rightarrow\mathbb{R}^{n+2}_{r+1}$ derives a $c-$polar transform of $y$ if it satisifies the following conditions: (i).$\langle\psi,Y\rangle=\langle\psi,Y_z\rangle=0$; (ii). $\psi_z\in Span_{\mathbb{C}}\{Y,Y_{\bar{z}}\}$.
\end{proposition}
\begin{proof}
From (ii), we see that $\langle\psi_{z},Y_{\bar{z}}\rangle=0$, together with (i), yielding $\langle\psi,Y_{{z}\bar{z}}\rangle=0$. So $\psi$ is a parallel section of the normal bundle by (ii). Also from (ii),
$\langle\psi_{z},\psi\rangle=0$, yielding $\langle\psi,\psi\rangle=c$ for some constant $c$.
\end{proof}

\begin{remark}  The Christoffel transform of an isothermic surface $x:M\rightarrow \mathbb{R}^n$ is defined as a map $x^C:M\rightarrow \mathbb{R}^n$ such that $x^C_z\parallel x_{\bar{z}}$ (see \cite{Bur}, \cite{Palmer}) . Such $x^C$ is also called dual to $x$.  For  $c-$polar transforms, we also can change the second condition such that  $\psi_z\in Span_{\mathbb{C}}\{Y_{\bar{z}}\}$ by scaling of $Y$. In this sense, we also can say that $\psi$ is dual to $Y$, and $c-$polar transforms can be looked as generalized Christoffel transforms.

On the other hand, (spacelike) isothermic surfaces, together with one of their Darboux transforms, are also characterized as curved flats or $O(m+1,1)/O(m)\times O(1,1)-$system II in integrable system theory, see for example \cite{BHPP}, \cite{BDPT}, \cite{CGS}, \cite{Du}, \cite{ZCC}. We also note that in \cite{BHPP}, the Christoffel transform of an isothermic surface is explained as a limit surface of Darboux transforms. While there does not exist an  integrable system theory structure for the Christoffel transform. Here the integrable system description of $c-$polar transforms is also unknown.
\end{remark}

Another notion related with $c-$polar transforms is the {\em generalized H-surface}, provided by Burstall as examples of isothermic surfaces.
\begin{definition}\cite{Bur} A surface $f:M\rightarrow N^n_r(c)$, with mean curvature vector $H$, is said to be a  {\em generalized H-surface} if it admits a parallel isoperimetric section. Here an {\em isoperimetric section} means a unit normal vector field $N$ of $f$ satisfying $\langle N, H \rangle\equiv const$. $N$ is called a {\em minimal section} if $\langle N, H \rangle\equiv 0$.
\end{definition}

Now we state our main theorem as follows:
\begin{theorem}\label{thm-polar} Let $y:M\rightarrow Q^{n}_{r}$  be a full spacelike isothermic surface. That is, the image of $y$ is not contained in any affine hyperplane.

(i). When $c\neq0$, on an open dense subset $M_0$ of $M$, any $c-$polar transform $\psi:M_0\rightarrow N^{n+1}_r(\frac{1}{c})$ is also a spacelike isothermic surface sharing the same adapted coordinate with $y$. Moreover, $\psi$ is a generalized H-surface admitting a parallel minimal section. To be concrete, there exists a lift $Y^\psi$ of $y$ with $Y^\psi_z\parallel \psi_{\bar{z}}$, that is, $Y^\psi$ is dual to $\psi$.

(ii). When $c=0$,  on an open dense subset $M_0$ of $M$, any $0-$polar transform $[\psi]:M_0\rightarrow Q^{n}_r$ of $y$ is also a spacelike isothermic surface sharing the same adapted coordinate with $y$. Moreover, $\psi$ is a generalized H-surface in $\mathbb{R}^{n+2}_{r+1}$ admitting a parallel minimal section. To be concrete, there exists a lift $Y^\psi$ of $y$ with $Y^\psi_z\parallel \psi_{\bar{z}}$, that is, $Y^\psi$ is dual to $\psi$.

(iii). The moduli space of isometric $c-$polar surfaces can be described as follows:
\begin{equation}\left\{\begin{split}
&\{c-\hbox{polar surfaces of }y\}\cong S^{n-3}_r(c)/Z_2=\{x\in \mathbb{R}^{n-2}_r|\langle x, x\rangle=c\}/Z_2, c>0 .\\
&\{0-\hbox{polar surfaces of }y\}\cong Q^{n-3}_{r-1}(c)=\{x\in \mathbb{R}^{n-2}_r|\langle x, x\rangle=0\}, c=0 .\\
&\{c-\hbox{polar surfaces of }y\}\cong H^{n-3}_{r-1}(c)/Z_2=\{x\in \mathbb{R}^{n-2}_{r}|\langle x, x\rangle=c\}/Z_2, c<0 .\\
\end{split}\right.
\end{equation}

(iv). Let $\psi$ be a $c-$polar transform of $y$. The conformal invariant metric of $\psi$ is of the form
\begin{equation}g^{\psi}=\left(\langle\kappa,\kappa\rangle+\left(\frac{\langle\psi,D_{z}\kappa\rangle}{\langle\psi,\kappa\rangle}\right)_{\bar{z}}\right)|dz|^2.
\end{equation}
Furthermore, suppose that $M$ is a closed surface. If $\psi$ is globally immersed, then $W(\psi)=W(y)$, i.e., in this case, the Willmore functional is polar transform invariant.
\end{theorem}
\begin{proof} Here we use the notions in Section 2. Let $\psi$ be a parallel section of normal bundle $V$ with $\langle\psi,\psi\rangle=c$. Then we have that
$$\psi_{z}=2\langle\psi,D_{\bar{z}}\kappa\rangle Y-2\langle\psi, \kappa\rangle Y_{\bar{z}}.$$
So
$$\langle\psi_{z},\psi_{\bar{z}}\rangle=2\langle\psi, \kappa\rangle^2.$$

Suppose that on an open subset $U_0$ of $M$, $\langle\psi, \kappa\rangle^2=0$. Then on $U_0$,  $\langle\psi, \kappa\rangle^2=\langle\psi, D_z\kappa\rangle^2=\cdots=0$. So $\psi_z=0$ and $y$ is in the hyperplane $\psi^{\perp}$, contradicting to the full property of $y$. So there exists an open dense subset $M_0$ of $M$ on which $\psi$ is an immersion. That is, $\langle\psi, \kappa\rangle^2\neq0.$ Then
\begin{equation}\label{eq-psi-zz}
\begin{split}
\psi_{zz}&=2\langle\psi,D_zD_{\bar{z}}\kappa\rangle Y-2\langle\psi, D_z\kappa\rangle Y_{\bar{z}}+2\langle\psi,D_{\bar{z}}\kappa\rangle Y_z-2\langle\psi, \kappa\rangle Y_{z\bar{z}}\\
&=\frac{2\langle\psi,D_{z}\kappa\rangle}{\langle\psi,\kappa\rangle}\psi_{z}+\Omega^{\psi},\\
\end{split}
\end{equation}
with
\begin{equation}\label{eq-psi-k}\begin{split}
\Omega^{\psi}=-&\langle\psi,\kappa\rangle N+2\langle\psi, D_z\kappa\rangle Y_{\bar{z}}+2\langle\psi,D_{\bar{z}}\kappa\rangle Y_z
\\
&+2(\langle\psi,D_zD_{\bar{z}}\kappa\rangle+\langle\psi, \kappa\rangle\langle\kappa, \kappa\rangle-\frac{2\langle\psi, D_z\kappa\rangle\langle\psi, D_{\bar{z}}\kappa\rangle }{\langle\psi, \kappa\rangle } ) Y.
\end{split}
\end{equation}
Since the normal bundle is flat and $\kappa$ is real-valued, we derive that
$$\Omega^{\psi}=\overline{\Omega^{\psi}}.$$

To show that $\psi$ is a generalized $H-$surface, we first compute
\begin{equation}\label{eq-psi-zb}
\psi_{z\bar{z}}=-2\langle\psi, \kappa\rangle\kappa+
(\langle\psi,2D_{\bar{z}}D_{\bar{z}}\kappa+\bar{s}\kappa\rangle)Y.
\end{equation}
Set $$Y^{\psi}=\frac{1}{\langle\psi,\kappa\rangle}Y.$$
We obtain that
$$Y^\psi_z=-\frac{1}{2\langle\psi,\kappa\rangle^2}\psi_{\bar{z}}.$$
We also have that
$$\langle Y^\psi,Y^\psi\rangle=\langle Y^\psi,\psi\rangle=\langle Y^\psi,\psi_z\rangle=\langle Y^\psi,\psi_{z\bar{z}}\rangle=0,\Rightarrow \langle Y^\psi,H^\psi\rangle=0.$$
Here $H^{\psi}$ denotes the mean curvature vector of $\psi$. Summing up, $Y^\psi$ is a parallel minimal section of $\psi$ and dual to $\psi$ (in the sense of \cite{Bur}).

Now we turn to the proof of (iii). Let $\tilde{\psi}$ be a $c-$polar surface isometric to $\psi$ with $\psi\mp\tilde\psi\neq0$. Then
$$\langle \psi_z,\psi_{\bar{z}}\rangle=\langle \tilde\psi_z,\tilde\psi_{\bar{z}}\rangle,\Rightarrow \langle \psi,\kappa\rangle=\pm\langle \tilde\psi,\kappa\rangle. $$
Similar to above, we see that $y$ falls into the affine hyperplane $(\psi\mp\tilde\psi)^{\perp}$, which is not allowed. So any two $c-$polar surfaces are not isometric. Then (iii) follows directly.

For the computation of $g^\psi$, by \eqref{eq-k-tr}, \eqref{eq-psi-k}, we have that
\begin{equation*}
\begin{split}
\langle\kappa^\psi,\bar\kappa^\psi\rangle
&=\frac{1}{4\langle\psi,\kappa\rangle^2}\langle\Omega^\psi,\Omega^\psi\rangle\\
&=\langle\kappa,\kappa\rangle+\frac{\langle\psi,D_{\bar{z}}
D_z\kappa\rangle\langle\psi,\kappa\rangle-\langle\psi,D_z\kappa\rangle\langle\psi,D_{\bar{z}}\kappa\rangle}{\langle\psi,\kappa\rangle^2}\\
&=\langle\kappa,\kappa\rangle+\left(\frac{\langle\psi,D_z\kappa\rangle}{\langle\psi,\kappa\rangle}\right)_{\bar{z}}.
\end{split}
\end{equation*}

As to the Willmore functional, it is just a consequence of Stokes formula and $$\left(\frac{\langle\psi,D_z\kappa\rangle}{\langle\psi,\kappa\rangle}\right)_{\bar{z}}dz\wedge d\bar{z}=-d\left(\frac{\langle\psi,D_z\kappa\rangle}{\langle\psi,\kappa\rangle}dz\right).$$
\end{proof}

\begin{remark} One can see that all $1-$polar surfaces are in the set $\{${\em Spacelike (full) isothermic generalized $H-$surfaces in  $S^{n+1}_r$ with null parallel minimal section }$ \} $. And for any spacelike (full) isothermic generalized $H-$surfaces in  $S^{n+1}_r$ with null parallel minimal section, it is dual to some spacelike isothermic surface in $Q^n_r$ and is also a $1-$polar transform of this surface.  In this sense, we may roughly say that
this subset of spacelike isothermic generalized $H-$surfaces in  $S^{n+1}_r$ is equivalent to the set $\{\hbox{Spacelike  isothermic surfaces in } Q^n_r\}$.
\end{remark}

\begin{corollary}Let $\psi$ and $\tilde\psi$ be two $c-$polar transforms of $y$. $\psi$ and $\tilde\psi$ have the same conformal invariant metric, if and only if
\begin{equation}F(\psi,\tilde\psi,z,\bar{z}):=\frac{\langle\psi,D_{z}\kappa\rangle\langle\tilde\psi,\kappa\rangle-
\langle\psi,\kappa\rangle\langle\tilde\psi,D_{z}\kappa\rangle}{\langle\psi,\kappa\rangle\langle\tilde\psi,\kappa\rangle}
\end{equation}
is a holomorphic function of $z$. As a consequence, if $F(\psi,\tilde\psi,z,\bar{z})$ is not a holomorphic function, they are not conformally equivalent to each other.
\end{corollary}

\begin{remark} It is in general unknown how to determine the conformal moduli space of $c-$polar transforms, i.e., the   $c-$polar transforms which are not conformally equivalent to each other. Below we provide examples with conformally equivalent polar transforms.
\end{remark}

\begin{example}(\cite{Ma-W1}).
In \cite{Ma-W1}, they constructed a class of homogenous spacelike tori which are both Willmore and isothermic. Set
$\varphi=\varphi(t,\theta)=\theta/\sqrt{t^{2}-1}$. Then
$Y_{t}(\theta,\phi):\mathbb{R}\times\mathbb{R}\rightarrow
\mathbb{R}^{6}_{2}$ is given by
\[
Y_{t}(\theta,\phi)=\left(\cos(t\varphi)\cos\phi,\cos(t\varphi)\sin\phi,
\sin(t\varphi)\cos\phi,\sin(t\varphi)\sin\phi,\cos\varphi,\sin\varphi\right).
\]
Note that the period condition is satisfied if $t$ is a rational
number; hence after projection $\pi$ we obtain an immersed torus. The polar transforms of $Y_t$ are  $[L_t]$ and $[R_t]$, which are globally immersed. Then by the theorem above, we see that
$$W(Y_t)=W(L_t)=W(R_t).$$
To be concrete, by generalizing the fundamental theorem of conformal surface theory in \cite{BPP}, one can show that $[L_t]$ and $[R_t]$ are both conformally equivalent to $[Y_t]$.
For further details see \cite{Ma-W1}.
\end{example}

Now we provide some other isothermic examples.
\begin{example} Set $I=[0,l]$ and $I=[0,\tilde{l}]$ . Let $\gamma:  I\rightarrow \mathbb{R}^{n_1}_{r_1}$ and $\tilde\gamma: \tilde{I}\rightarrow \mathbb{R}^{n_2}_{r_2}$ be two spacelike curves with arc parameter $t,\tilde{t}$ respectively.
Then $$x=(\gamma,\tilde\gamma):I\times\tilde{I}\rightarrow \mathbb{R}^n_r,\ n=n_1+n_2,\ r=r_1+r_2,$$
is a spacelike isothermic surface with adapted coordinate $z=t+i\tilde{t}$.

 For simplicity, we assume that $\gamma$ and $\tilde\gamma$ are both in $\mathbb{R}^2$. And set
$$\gamma_t=\alpha,\ \alpha_t=k\beta,\ \tilde\gamma_{\tilde{t}}=\tilde\alpha,\ \tilde\alpha_{\tilde{t}}=\tilde{k}\tilde\beta,$$
with $(\alpha,0),(\beta,0),(0,\tilde\alpha),(0,\tilde\beta)$ an orthonormal basis of $\mathbb{R}^4$. The mean curvature of $x$ is $H=\frac{1}{2}(k\beta,\tilde{k}\tilde{\beta})$. And $\{(\beta,0),(0,\tilde\beta)\}$ is a parallel orthonormal basis of normal bundle. So we see that $x$ is generalized $H-$surface if and only if one of $k$, $\tilde{k}$ is constant.

On the other hand, any $1-$polar transforms of $x$ is of the form
$$\psi^\theta=(-\langle n^\theta,x\rangle,n^\theta,\langle n^\theta,x\rangle)+\frac{1}{2}(k\cos\theta+\tilde{k}\sin\theta)(\frac{1-\langle x,x\rangle}{2},x,\frac{1+\langle x,x\rangle}{2})
$$
with $n^\theta=(\beta\cos\theta,\tilde\beta\sin\theta)$ for any $\theta\in[0,\pi)$. And by computation, $\langle\psi^\theta,\kappa\rangle=k\cos\theta-\tilde{k}\sin\theta$. So $\psi^0$ and $\psi^{\frac{\pi}{2}}$ must be not conformally equivalent to each other when $k_t^2+\tilde{k}_{\tilde{t}}^2\neq0$. And if both $k$ and $\tilde{k}$
are constant, $x$ is just in a $3-$dimensional space form and there exists only one $c-$polar surface of $x$.

If we assume furthermore that both $\gamma$ and $\tilde{\gamma}$ are closed curves with $k\tilde{k}>0$, that is, $x$ is an immersed 2-torus, then for any $\theta$ such that $k\cos\theta-\tilde{k}\sin\theta\neq0$ on $T=x(I\times\tilde{I})$, $\psi^\theta$ is also an immersed 2-torus with $W(\psi^\theta)=W(x)$. Note that a special example of $x$ is just the Clifford torus with $\gamma=\tilde\gamma=S^1(\rho^2)\subset\mathbb{R}^2$, which is conformally equivalent to any $c-$polar surface of it.
\end{example}

\section{Two permutability theorems}

Similar to the proof in \cite{Ma-W2}, we also can obtain the permutability theorems for $c-$polar transforms with
spectral transforms and Darboux transforms. For simplicity, here we just give a rough proof.

\subsection{Permutability with spectral transform}

For an immersed spacelike
isothermic surface $y:M \rightarrow Q^{n}_{r}$, the
conformal Gauss, Codazzi, and Ricci equations are still satisfied
under the deformation
\[
s^{\tilde{c}}=s+\tilde{c},~\langle\kappa^{\tilde{c}},\psi^{\tilde{c}}\rangle=\langle\kappa,\psi\rangle,
~\langle\psi^{\tilde{c}},\psi^{\tilde{c}}\rangle=\langle\psi,\psi\rangle,
~D_z^{\tilde{c}}=D_z,
\]
where $\tilde{c}\in\mathbb{R}$ is a real parameter, and $\psi$ is arbitrary section of normal bundle with deforming normal section $\psi^{\tilde{c}}$. By
the integrable conditions, there are an associated family of
non-congruent isothermic surfaces $[Y^c]$ with corresponding
invariants. They are
called the \emph{spectral transforms} of the original surface (see
\cite{BPP}).

\begin{theorem}\label{thm-commu1}
Let $y^{\tilde{c}}$ be a spectral transform (with parameter $\tilde{c}$) of
$y:M\rightarrow Q^{n}_{r}$, both being spacelike isothermic
surfaces. Denote their canonical lift as $Y, Y^{\tilde{c}}$ for the same
adapted coordinate $z$. If a $c-$polar surface $\psi$ is non-degenerate, then there exists a $c-$polar surface $\psi^{\tilde{c}}$ of $y^{\tilde{c}}$ such that $\psi^{\tilde{c}}$ is
also a spectral transform (with parameter $c$) of $\psi$, i.e., we
have the commuting diagram:
\begin{equation*}
\begin{CD}
[Y]@>{}>> [Y^{\tilde{c}}] \\
@V{}VV @V{}VV\\
\psi @>{}>> ~\psi^{\tilde{c}}\
\end{CD}
\end{equation*}
\end{theorem}

\begin{proof}Since the transform preserve normal connection and $\kappa$,
let $\psi^{\tilde{c}}$ be the corresponding normal section of $\psi$, then $\psi^{\tilde{c}}$ is a $c-$polar transform of $y^{\tilde{c}}$. From the equations \eqref{eq-k-tr},\eqref{eq-psi-k},\eqref{eq-psi-zz},\eqref{eq-psi-zb}, we see that the other conditions except Schwarzian are all satisfied. As to $s$, one computes $$s^\psi=2\omega^\psi_{zz}-2(\omega^\psi_z)^2+\frac{\langle\psi, D_{\bar{z}}D_{\bar{z}}\kappa\rangle}{\langle\psi, \kappa\rangle }+s.$$
So we have $s^{\psi^{\tilde{c}}}=s^\psi+\tilde{c}$. We conclude that $\psi^{\tilde{c}}$ is  a spectral transform
of $\psi$ with parameter $\tilde{c}$.
\end{proof}

\subsection{Permutability with Darboux transform}

We recall the basic properties of Darboux transforms as below. For more details, see \cite{BDPT}, \cite{Bur}, \cite{Jer},
\cite{Ma1}, \cite{Ma2006}, \cite{Ma-W2}.
\\
\\
\noindent {\bf Definition and Proposition}  \cite{Ma-W2} {\it Let
$y:M\rightarrow Q^{n}_{r}$ denote a spacelike isothermic surface
with canonical lift $[Y]$ with respect to the adapted coordinate
$z$. A spacelike immersion $y^{\ast}: M\rightarrow Q^{n}_{r}$ is
called a {\rm Darboux transform} of $y$ if its local lift
$Y^{\ast}$ satisfies
\begin{equation}\label{eq-isoth}
\langle Y, Y^*\rangle\neq0,\ Y^{\ast}_{z}\in {\rm Span}_{\mathbb{C}}\{Y^{\ast},Y,Y_{\bar{z}}\}.
\end{equation}
Note that this is well-defined, where $Y^{\ast}$ is not
necessarily the canonical lift. We have the following conclusions:
\par 1) $y,y^{\ast}$ are conformal; they envelop one and the same
round 2-sphere congruence given by ${\rm Span} \{Y,Y^{\ast},{\rm
d}Y)\}={\rm Span} \{Y,Y^{\ast},{\rm d}Y^{\ast}\}$.
\par 2) Set $\langle Y,Y^{\ast}\rangle =-1$. Then
$Y^{\ast}_z=\frac{\mu}{2}Y^{\ast}+\theta(Y_{\bar{z}}+\frac{\bar\mu}{2}Y)$,
where $\theta$ is a non-zero real constant. This Darboux transform
is specified as \emph{$D^{\theta}$-transform}.
\par 3) $Y^{\ast}$ is an isothermic surface sharing the same
adapted coordinate $z$. Hence the curvature lines of
~$y,\tilde{y}$ do correspond.}

For the proof, we refer to \cite{Ma-W2}. Now we have the permutability theorem as below.

\begin{theorem}\label{thm-commu2}
Let $y:M\rightarrow Q^{n}_{r}$ be a spacelike isothermic surface
and $[Y^{\ast}]$ be a $D^\theta$-transform of $y$.
If $\psi$ is a non-degenerate $c-$polar surface of $y$, then there exists a $c-$polar surface $\psi^{\ast}$ of $y^{\ast}$ such that $\psi^{\ast}$ is
also a $D^\theta$-transform of $\psi$, i.e., we
have the commuting diagram:\begin{equation}
\begin{CD}
[Y] @>D^\theta-{\rm transform}>> [Y^{\ast}] \\
@V{left\ c-polar}VV @VV{left\ c-polar}V\\
\psi @>D^\theta-{\rm transform}>> ~\psi^{\ast}\
\end{CD}
\end{equation}
\end{theorem}
\begin{proof}

We may assume that
\[
Y^{\ast}=N+\bar{\mu}Y_z +\mu
Y_{\bar{z}}+(\frac{1}{2}|\mu|^2-\langle\xi,\xi\rangle)Y+\xi.
\]
Then differentiating $Y^{\ast}$, we obtain
\begin{equation}\theta=\mu_z-s-\frac{\mu^2}{2}-2\langle\kappa,\xi\rangle.\end{equation}
One can verify that $\theta$ is a constant via the integrable equations, following the methods in \cite{Ma1}, \cite{Ma2006}, \cite{Ma-W2}.

Set
\begin{equation}P=Y_z+\frac{\mu}{2}Y,\end{equation}
the structure equations can be written as
\begin{equation}\label{eq-moving-space-iso-1}
\left\{\begin {array}{lllll}
Y_{z}=-\frac{\mu}{2}Y+P,\\[1mm]
Y^{\ast}_{z}=~\frac{\mu}{2}Y^{\ast}+\theta \bar{P},\\[1mm]
P_z=~\frac{\mu}{2}P+(\frac{\theta}{2}+\langle\kappa,\xi\rangle)Y+\kappa,\\[1mm]
\bar{P}_{z}=-\frac{\mu}{2}\bar{P}+\frac{1}{2}Y^{\ast}+\langle\xi,\xi\rangle Y-\frac{1}{2}\xi,\\[1mm]
\psi_{z}=-2\langle\psi,\kappa\rangle\bar{P}-\langle\psi,D_{\bar{z}}\xi\rangle Y.
\end {array}\right.
\end{equation}

Set \begin{equation}\psi^{\ast}=\psi+\langle \psi,\xi\rangle Y-\frac{\langle \psi,\kappa\rangle}{\theta}Y^{\ast}.\end{equation}
It is direct to verify that $$\psi^{\ast}\perp\{Y^{\ast},Y^{\ast}_z,Y^{\ast}_{z\bar{z}} \},$$
and $\psi^{\ast}$ is a parallel section of $Y^{\ast}$ with length $c$. By similar computations as \cite{Ma-W2}, one can prove
 that $\psi^{\ast}$ is a Darboux transform of $\psi$
with the same parameter $\theta$.
\end{proof}

{\bf Acknowledgments}
The author is thankful to Professor Xiang Ma for valuable discussions and suggestions.

\def\refname{Reference}

\vspace{5mm}  \noindent Peng Wang, {\small Department of Mathematics, Tongji University,
Siping Road 1239, Shanghai, 200092, People's Republic
of China. e-mail: {\sf netwangpeng@mail.tongji.edu.cn}
\end{document}